\documentclass[10 pt]{amsart}
\usepackage[active]{srcltx}
\usepackage[mathscr]{eucal}
\usepackage{amsmath,amssymb,hyperref}
\usepackage{graphics}

\newtheorem{thm}{Théorème}[section]
\newtheorem{THM}{Théorème}
\newtheorem{PROP}{Proposition}[section]

\newtheorem{cor}[thm]{Corollaire}
\newtheorem{prop}[thm]{Proposition} 
\newtheorem{lemme}[thm]{Lemme}
\newtheorem{definition}[thm]{D\'efinition}
\newtheorem{rem}[thm]{Remarque}

\input xy
\xyoption{all}
\usepackage[french]{babel}
\usepackage[T1]{fontenc}
\usepackage[latin1]{inputenc}

\usepackage{amsfonts,eufrak}

\usepackage{amssymb, amsmath}

\newcommand{\field}[1]{\mathbb{#1}}
\newcommand{\CC}{\field{C}}
\newcommand{\RR}{\field{R}}
\newcommand{\QQ}{\field{Q}}

\newcommand{\DD}{\field{D}}

\begin{document}
 
\title[Uniformisation de l'espace des feuilles]{Uniformisation de l'espace des feuilles de certains feuilletages de codimension un}
\vskip 10 pt





\selectlanguage{french}

\author{FR\'ED\'ERIC TOUZET$^1$}
\thanks{$^1$ IRMAR, Campus de Beaulieu 35042 Rennes Cedex, France, frederic.touzet@univ-rennes1.fr}

\begin{abstract} Dans cet article, nous étudions, sur des variétés Kähler compactes, les feuilletages holomorphes (éventuellement singuliers) dont le fibré conormal est pseudo-effectif. En utilisant la notion de courant à singularités minimales, nous montrons que l'on peut munir canoniquement l'espace des feuilles d'une métrique à courbure constante négative ou nulle dont les éventuelles dégénerescences sont localisés le long d'une hypersurface invariante ``rigidement plongée'' dans la variété.\\ 
 
{\sc Abstract.} This paper deals with codimension one (may be singular) foliations on compact Kälher manifolds whose conormal bundle is assumed to be pseudo-effective. Using currents with minimal singularities, we show that one can endow the space of leaves with a metric of constant non positive curvature wich may degenerate on a ``rigidly'' embedded invariant hypersurface.

\end{abstract}

\maketitle

\section{Notations et rappels}
Soit $M$ une vari\'et\'e k\"ahlerienne compacte et connexe. 

Une {\it distribution holomorphe de codimension 1}  (\'eventuellement singuli\`ere) sur $M$ correspond \`a la donn\'ee d'une section $\omega$ {\it holomorphe non triviale} de $\Omega^1 (M)\otimes L$ o\`u $L$ est un  fibr\'e en droites holomorphe. D'un point de vue dual, on peut également définir un tel objet par un sous faisceau saturé $\mathcal F$ de rang $n-1$ ($n=Dim_{\CC}M$) de $TM$ (en l'occurence, le sous faisceau annulateur de $\omega$).

  Le th\'eor\`eme suivant, d\^u \`a Jean-Pierre Demailly précise les hypoth\`eses ``minimales'' que l'on peut faire sur le fibr\'e $L$ pour que cette distribution d\'efinisse un feuilletage holomorphe de codimension $1$.

\medskip

\begin{thm}\label{de} \cite{de} Soit une distribution holomorphe singuli\`ere donn\'ee comme ci-dessus; supposons que le dual $L^*$ du fibr\'e $L$ soit pseudo-effectif; alors la forme $\omega$ est intégrable.
\end{thm}

Rappelons \`a cet effet qu'une classe de cohomologie $\alpha\in H^{1,1}(M,\RR)$ est dite {\it pseudo-effective} si $\alpha$ peut \^etre repr\'esent\'ee par un courant positif ferm\'e de bidegr\'e $(1,1)$.

On dira alors qu'un fibr\'e  en droites holomorphe $L$ est {\it pseudo-effectif} si $c_1(L)$ est pseudo-effective ou,de fa\c con \'equivalente,  si l'on peut implanter sur $L$ une m\'etrique $h(x,v)= {|v|}^2e^{-2\varphi (x)}$ o\`u le poids local $\varphi$ est une fonction plurisousharmonique ({\it psh} pour faire bref). Le courant $T$ est alors \'egal \`a la forme de courbure d'une telle m\'etrique (\'eventuellement singuli\`ere) suivant la formule:
\begin{equation}\label{courbure}
 T=\frac{-i}{2\pi}\partial\overline{\partial}\log h=\frac{i}{\pi}\partial\overline{\partial}\varphi.
\end{equation}

Dans l'\'enonc\'e du th\'eor\`eme \ref{de}, on retrouve un r\'esultat classique lorsque $L^*={\mathcal O}(D)$ o\`u $D$ est un diviseur effectif: 
\medskip 
\begin{thm}  Toute forme holomorphe sur une vari\'et\'e k\"alherienne compacte est ferm\'ee (donc en particulier int\'egrable).
\end{thm}

Dans ce qui suit, on se propose de pr\'eciser la structure des distributions ({\it a posteriori} intégrables) v\'erifiant les hypoth\`eses du th\'eor\`eme de Demailly . Soit $s\in H^0(M,\Omega^1 (M)\otimes L)\setminus\{0\}$ représentant une telle distribution. Puisque la somme d'une classe effective et d'une classe pseudo-effective reste pseudo-effective, on peut toujours se ramener au cas où $s$ est r\'eguli\`ere en codimension 1, i.e: son lieu d'annulation est un sous-ensemble analytique de codimension $\geq 2$ qui coïncide alors avec le lieu singulier  $ Sing\ {\mathcal F}$ de la distribution. Le fibr\'e $L$ représente alors  le {\it fibr\'e normal} $N_{\mathcal F}$ de la distribution et, par hypoth\`ese, son dual, le fibr\'e conormal  ${N_{\mathcal F}}^*$, est pseudo-effectif. 

Dans la suite il sera également commode de regarder alternativement ${N_{\mathcal F}}^*$ comme le sous faisceau inversible de $\Omega^1 (M)$ dont les sections locales sont les formes s'annulant en restriction à la distribution.\\

Introduisons maintenant quelques notations.

Si $T$ es un $(1,1)$ courant positif fermé, on désignera par $\{T\}$ sa classe de cohomologie dans $H^{1,1}(M,\RR)$. 

Soit $d=\sum_{D\in Div(M)}{\lambda_D}D$ un élément de $Div(M)\otimes\RR$ auquel on peut associer le courant d'intégration $T_d=\sum_{D\in\RR}{\lambda_D}[D]$ .

On posera alors $\{d\}:=\{T_d\}$.\\
 

Donnons une formulation plus pr\'ecise du th\'eor\`eme \ref{de}.

Soit $T=\frac{i}{\pi}\partial\overline{\partial}\varphi$ un courant positif dont la classe de cohomologie $\{T\}$ dans le groupe de N\'eron-S\'everi réel  est égale à $c_1( {N_{\mathcal F}}^*)$.

Ceci permet d'exhiber une $(1,1)$ forme positive globalement d\'efinie
\begin{equation}\label{eta}
\eta_T=\frac{i}{\pi}e^{2\varphi} \omega\wedge \overline{\omega}
\end{equation}
à coefficients $L^\infty_{loc}$ o\`u la $1$ forme holomorphe $\omega$ est un générateur local de $\mathcal F$. Il est facile de constater que $\eta_T$ est bien définie à multiplication près par un réel positif.

L'int\'egrabilit\'e de $\mathcal F$ r\'esulte alors imm\'ediatement de l'\'egalit\'e suivante, \'etablie dans \cite{de}:
  \begin{equation}\label{integ}
d\omega=-\partial\varphi\wedge\omega
\end{equation}

et qui entra\^\i ne les relations
\begin{equation}\label{etafermé}
d\eta_T=0
\end{equation}
(au sens des courants) et

\begin{equation}\label{Tomega=0}
T\wedge \omega=0
\end{equation}

\noindent dont nous ferons usage par la suite.

Pour fixer les idées, précisons comment on peut établir la relation (\ref{integ}) en admettant l'intégrabilité de $\mathcal F$ pour laquelle on renvoie le lecteur à \cite{de} ou encore \cite{brpt}. 

Soit $\theta$ une forme de Kähler sur $M$. Considérons le courant $\Omega=i\partial\overline{\partial}(\eta_T\wedge {\theta}^{n-2})$ où $n$ est la dimension complexe de $M$.
Au voisinage d'un point non situé dans l'ensemble singulier $\mathcal Sing\ \mathcal F$ du feuilletage, on peut écrire en coordonnées holomorphes locales, $\omega=fdz$ où $f$ est holomorphe inversible. Sur $U=M\setminus \mathcal S$, on obtient donc que $$\Omega=- \partial\overline{\partial}(e^{2\varphi+2\log{|f|}})dz\wedge d\overline{z}$$
 est positif (l'exponentielle d'une fonction {\it psh} est {\it psh}). Par ailleurs, $i\eta_T\wedge {\theta}^{n-2}$ est un courant positif à coefficients localement bornés; il est donc dominé par $c\theta^{n-1}$ pour $c>0$ assez grand. En utilisant alors le principal résultat de \cite{albas} et le fait que $Sing\ \mathcal F$ est de codimension au plus deux, on conclut que $\Omega$ est une mesure {\it positive} sur $M$ tout entier et que finalement $\Omega=0$ par exactitude. 

On peut donc en déduire qu'au voisinage d'un point régulier, $e^{2\varphi+2\log{|f|}}$ est {\it pluriharmonique} dans les feuilles et finalement constante car l'exponentielle d'une fonction {\it psh} est pluriharmonique si et seulement si cette fonction est {\it constante}.

L'identité (\ref{integ}) est donc vérifiée sur le complémentaire de $Sing\ \mathcal F$. En appliquant le $\overline{\partial}$, on obtient 
$\partial\overline{\partial}\varphi\wedge\omega=0$ sur $M\setminus Sing\ \mathcal F$ et donc sur $M$ en utilisant que $i\partial\overline{\partial}\varphi$ est positif et codim $Sing\ \mathcal F\geq 2$.  Par hypoellipticité du $\overline{\partial}$, on peut alors conclure que (\ref{integ}) est vraie sur $M$.

Avant d'\'enoncer notre principal r\'esultat, il nous semble utile de rappeler les notions de courants \`a singularit\'es minimales, de courants invariants par un feuilletage et de faire le lien entre ces différents objets.

\section{Courants invariants par holonomie et décomposition de Zariski}

\begin{definition}Soit $\mathcal F$ un feuilletage holomorphe de codimension 1 (\'eventuellement singulier) sur une vari\'et\'e $M$ complexe et $T$ un courant positif ferm\'e de bidegr\'e $(1,1)$ d\'efini sur $M$. On dira que $T$ est $\mathcal F$ invariant (ou invariant par holonomie de $\mathcal F$) si au voisinage de tout point de $M$, on a
       $$T\wedge \omega=0$$
o\`u $\omega$ d\'esigne une 1 forme holomorphe d\'efinissant localement le feuilletage.

\end{definition}

Au voisinage d'un point r\'egulier du feuilletage o\`u celui-ci est d\'ecrit par l'équation $\{dz=0\}$, un tel courant s'exprime donc sous la forme $T=\sqrt{-1}a(z)dz\wedge d\overline{z}$ o\`u $a$ est une mesure positive.

Dans la mesure o\`u nous cherchons \`a produire des m\'etriques tranverses au feuilletage invariante par holonomie \`a courbure constante  (dans un sens qui reste \`a pr\'eciser) et pr\'esentant la plus forte r\'egularit\'e possible, il est assez naturel de consid\'erer la classe plus restreinte des courants \`a singularit\'es minimales dont nous rappelons maintenant la d\'efinition (voir par exemple \cite{bo}).\\

Soit $M$ une vari\'et\'e kählerienne compacte et  $\gamma$ une $(1,1)$ forme ferm\'ee lisse. Dans une classe de cohomologie fix\'ee $\alpha\in H^{1,1}(M,\RR)$, consid\'erons l'ensemble $\alpha[-\gamma]$ des $(1,1)$ courants ferm\'es $T$ tels que $T\geq -\gamma$. On a donc
\begin{equation}
 T=\beta+i\partial\overline{\partial}\varphi
\end{equation}

\noindent o\`u $\beta$ est une $(1,1)$ forme ferm\'ee {\it lisse} telle $\{\beta\}=\{T\}$ et $\varphi$ une fonction {\it quasi-plurisousharmonique}, c'est-à-dire qui s'exprime localement comme la somme d'une fonction ${\it psh}$ et d'une fonction lisse. Ceci permet de d\'efinir, selon la formule habituelle (voir \cite{bo}), le nombre de Lelong $\nu (T,x)$ de $T$ en n'importe quel point $x\in M$ et le nombre de Lelong g\'en\'erique $\nu (T,Y)$ sur un sous-ensemble analytique $Y$.

Soient maintenant deux courants $T_1,T_2\in \alpha[-\gamma]$. on peut donc \'ecrire

   \begin{equation}
 T_1=\beta_1+i\partial\overline{\partial}\varphi_1,\ T_2=\beta_2+i\partial\overline{\partial}\varphi_2
\end{equation}

\begin{definition}  On dira que $T_1$ est moins singulier que $T_2$ et on note $T_1\preceq T_2$ s'il existe une constante r\'eelle $C$ telle que

$$\varphi_1\leq\varphi_2 +C$$

Notons que cette propriété  est ind\'ependante des choix de $\beta_i,\varphi_i$ et que $\preceq$ d\'efinit donc une relation pr\'e-ordre  sur $\alpha[-\gamma]$.
 \end{definition}

\begin{prop} (cf \cite{bo})   $(\alpha[-\gamma],\preceq$ admet un plus petit \'el\'ement.

\end{prop}

Ce plus petit \'el\'ement est appel\'e courant {\it \`a singularit\'es minimales} (dans la classe $(\alpha[-\gamma])$. Il n'est bien entendu pas unique, cependant, deux courants \`a singularit\'es minimales ont exactement  m\^eme nombre de Lelong en chaque point.

Plus g\'en\'eralement, il r\'esulte directement des d\'efinitions que si $T,T'\in (\alpha[-\gamma],\preceq)$ avec $T$ \`a singularit\'es minimales, on a pour tout $x\in M$

\begin{equation}
\nu (T,x)\leq \nu(T',x)
\end{equation}

Par ailleurs, lorsque $\gamma$ est nulle, on parlera de courant {\it positif \`a singularit\'es minimales}.

\begin{definition}\label{nefcod1} (\cite{bo})
 Soit $\alpha$ une classe pseudo-effective, on dira que $\alpha$ est {\bf nef en codimension 1} ou {\bf nef modifi\'ee}  si, pour une forme de K\"ahler fix\'ee $\theta$, il existe pour tout $\varepsilon>0$, un courant $T_\varepsilon\in\alpha[-\varepsilon\theta]$ lisse en dehors d'un sous-ensemble analytique de codimension $\geq 2$. 
\end{definition}

Cette d\'efinition est \'evidemment ind\'ependante de la métrique consid\'er\'ee.

Suivant (\cite{bo}), on notera $\mathcal M\mathcal N$ le c\^one formé par les classes nefs modifi\'ees.

\begin{definition}\label{famexcep}(\cite{bo})
 Soit $\mathcal A$ une famille finie de $p$ diviseurs premiers distincts: ${\mathcal A}=\{D_1,...,D_p\}$. On dira que cette famille est {\bf exceptionnelle} si le cône convexe engendré par les $\{D_i\}$ n'intersecte le cône nef modifié $\mathcal M\mathcal N$ qu'à l'origine.
\end{definition}

\begin{thm}\label{deczar} (\cite{bo})  Soit $\alpha$ une classe pseudo-effective, il existe alors une unique d\'ecomposition de la  forme 
       \begin{equation}\label{zar}
        \alpha=\{N(\alpha)\}+Z(\alpha)
       \end{equation}
o\`u $Z(\alpha)$ est {\bf nef en codimension 1} et $N(\alpha)$ ({\bf la partie n\'egative}) est un $\RR$ diviseur effectif  dans $Div(M)\otimes \RR$:

       \begin{equation}
        N(\alpha)=\sum_{i=1}^p\lambda_i D_i
       \end{equation}

telle que la famille $\mathcal A=\{D_1,...,D_p\}$ soit exceptionnelle.

Cette décomposition est par ailleurs unique 
et les coefficients $\lambda_i$ sont d\'etermin\'es suivant la formule :

       \begin{equation}\label{deflambdai}
        \lambda_i=sup_{\varepsilon>0}\nu(T_{min,\varepsilon}, D_i)
       \end{equation}
o\`u $T_{min,\varepsilon}$ d\'esigne un courant \`a singularit\'es minimales dans la classe $\alpha[-\varepsilon\theta]$.
\end{thm}

\begin{definition}\label{dimnum}(\cite{bo})
 La {\bf dimension numérique} $\nu (\alpha)$ de la classe pseudo-effective $\alpha$ est le plus petit entier $n\in\{0,...,Dim\ M\}$ tel que $Z^{n+1}(\alpha)=0$.
\end{definition}

Passons en revue quelques propri\'et\'es fondamentales de cette d\'ecomposition, appel\'ee {\it d\'ecomposition divisorielle de Zariski}.

\begin{prop} (\cite{bo})\label{propriétészar}
 Soit $\alpha$ pseudo-effective et $N(\alpha)=\sum_{i=1}^p\lambda_i D_i$ sa  partie n\'egative; alors

i) La d\'ecomposition divisorielle $\alpha=\{N(\alpha)\}+Z(\alpha)$ co\"\i ncide avec la d\'ecomposition de Zariski classique lorsque $M$ est une surface; en particulier, elle est dans ce cas rationnelle (i.e $N(\alpha)$ est un $\mathbb Q$ diviseur effectif si $\alpha$ est une classe rationnelle).\\

ii) La famille $(\{D_1\},...,\{D_p\})$ est libre dans $H^{1,1}(M,\RR)$. En particulier,  $p$ est au plus \'egal au nombre de Picard $\rho (M)$ de la variété.\\

iii)  Pour tout $i$, on a $\lambda_i\leq \nu (T,D_i)$ quel que soit le courant ferm\'e positif tel que $\{T\}=\alpha$. En particulier, $\sum_{i=1}^p\lambda_i[D_i]$ est le seul courant positif représentant $\{N(\alpha)\}$.

\end{prop}

\begin{rem}
 
On peut avoir $\lambda_i< \nu (T,D_i)$ dans iii) m\^eme lorsque $T$ est à singularit\'es minimales (voir \cite{bo}). Nous verrons cependant que l'\'egalit\'e est atteinte dans la situation qui est la n\^otre.
\end{rem}

Reprenons nos hypoth\`eses initiales: la vari\'et\'e ambiante est k\"ahlerienne compacte et muni d'un feuilletage holomorphe $\mathcal F$ dont le fibr\'e conormal $N_{\mathcal F}^*$ est {\it pseudo-effectif}.\\

D'apr\`es les formules (\ref{etafermé}) et (\ref{Tomega=0}), nous disposons de deux courants positifs $\mathcal F$ invariants, \`a savoir $T$ et $\eta_T$, le premier apparaissant comme la ``forme de courbure''de la m\'etrique transverse d\'efinie par le second (au sens de la relation (\ref{courbure})).

Puisque $\eta_T$ est \`a coefficients localement born\'es, la classe $\{\eta_T\}$ est {\it nef} et pour tout $(1,1)$ courant positif fermé $\Xi$, le produit ext\'erieur

$$\Xi\wedge \eta_T$$ est bien d\'efini en tant que $(2,2)$ courant positif par la formule usuelle (puisque $\xi$ est d'ordre $0$ et $\eta_T$ à coefficients $L_{loc}^\infty$) et coïncide avec $\frac{i}{\pi}\partial\overline{\partial}(\varphi\Xi) $, $\varphi$ un potentiel local de $\eta_T$
   (en utilisant des noyaux régularisants et le fait que  les potentiels locaux $\varphi_\eta$ de $\eta$ sont également  $L_{loc}^\infty$  ).Cette opération est par conséquent compatible avec le cup-produit en cohomologie:
\begin{equation}\label{produit coh}
  \{\Xi\wedge \eta_T\}=\{\Xi\}\{\eta_T\}
\end{equation}

Par ailleurs, on a 
\begin{equation}\label{nulXiwedge eta}
 \Xi\wedge \eta_T=0
\end{equation}

lorsque $\Xi$ est de plus $\mathcal F$ invariant (par exemple, $\Xi=T$).

En particulier,
\begin{equation}\label{nuletawedge eta}
  {\{\eta_T\}}^2=0
\end{equation}

\begin{prop}\label{Z(alpha)nef} Soit $\Xi$ un $(1,1)$ courant positif $\mathcal F$ invariant (par exemple, $\Xi=T$). Consid\'erons la d\'ecomposition de Zariski
 
$$  \alpha =\{N(\alpha)\}+Z(\alpha)$$

avec $\alpha=\{\Xi\}$ (par exemple, $\alpha=c_1(N_{\mathcal F}^*)=\{T\})$;

On a alors les propriétés suivantes.

\par i) Les composantes $D_i$ de la partie négative $N(\alpha)$ sont des séparatrices du feuilletage;  en particulier,  $Z(\alpha)$ peut être représenté par un courant positif $\Xi_0$ également $\mathcal F$ invariant.

\par  ii) $Z(\alpha)$ est proportionnelle \`a $\{\eta_T\}$ 

\par  iii) La  classe $Z(\alpha)$ est {\it nef} et  ${Z(\alpha)}^2=0$

      \par iv) La décomposition est orthogonale: $\{N(\alpha)\} Z(\alpha)=0$.
       \par v) Dans $H^{1,1}(M,\RR)$, le  $\RR$ espace vectoriel engendré par les composantes $\{D_i\}$ de $\{N(\alpha)\}$ n'intersecte la droite vectorielle réelle engendré par $\eta_T$ qu'à l'origine.
       \par vi) La décomposition est  rationnelle si $M$ est projective et $\alpha$ est une classe rationnelle (par exemple, $\alpha=c_1(N_{\mathcal F}^*)$).
       \par vii)Tout  $(1,1)$ courant positif fermé représentant $\alpha$ est nécessairement $\mathcal F$ invariant.
       \par viii) Soit $H$ une hypersurface invariante par le feuilletage et $H_1,...,H_r$ ses composantes irréductibles; la famille $\{H_1,...,H_r\}$ est exceptionnnelle si et seulement si la matrice $(m_{ij})=(\{H_i\}\{H_j\}{\{\theta\}}^{n-2})$ est définie négative, $\theta$ désignant une forme de Kähler fixée sur $M$.
\end{prop}

\begin{rem}
La propriété $\nu(c_1(N_{\mathcal F}^*))\leq 1$ énoncée dans iii) n'est rien d'autre qu'une version numérique du théorème de Bogomolov-Castelnuovo-De Franchis:
\end{rem}

\begin{thm} (\cite{Reid}) Soit $M$ une variété projective et $\mathcal F$ un feuilletage de codimension $1$;
alors $\kappa(N_{\mathcal F}^*)\leq 1$  ($\kappa$ désignant la dimension de Kodaira) et  en cas d'égalité, $\mathcal F$ est une fibration au-dessus d'une courbe .
\end{thm}

Signalons qu'il existe des feuilletages à conormal pseudo-effectif tel que $\kappa(N_{\mathcal F}^*)=-\infty$ dont les prototypes sont les feuilletages minimaux (c'est-à-dire à feuilles denses) sur certaines surfaces de type général uniformisées par le bidisque (voir \cite{brsurf}).

De façon générale, on verra de plus (section \ref{dyn}) que le comportement dynamique du feuilletage est (du moins en dimension numérique 1) fidèlement reflété par les propriétés algébriques du conormal; à savoir que $\mathcal F$ sera minimal (en un certain sens) si et seulement si $\kappa(N_{\mathcal F}^*)< 1$; {\it a contrario} le feuilletage sera une fibration si et seulement si le principe d'abondance est vérifié, ce qui n'est d'ailleurs qu'une reformulation du théorème précédent.
\bigskip

On peut d\'emontrer la proposition \ref{Z(alpha)nef} via le th\'eor\`eme de la signature de Hodge-Riemann dans un esprit similaire \`a ce qui est fait dans \cite{ds}.

Considérons à cet effet la forme bilinéaire symétrique $q$ définie sur $H^{1,1}(M,\mathbb R)$ par

  $$q(c,c')=-\int_M c\wedge c'\wedge{\{\theta\}}^{n-2},$$

$\theta$ étant une forme de Kähler fixée sur $M$.

\begin{thm} (Hodge-Riemann)

Notons $\mathcal P$ l'ensemble des classes primitives dans $H^{1,1}(M,\mathbb R)$:

$$\mathcal P:=\{c\in H^{1,1}(M,\mathbb R)|c\wedge {\{\theta\}}^{n-1}=0\}.$$

Alors la forme $q$ est définie positive sur l'hyperplan $\mathcal P$.
\end{thm}

Afin de détailler la preuve de la proposition précédente, voici d'abord une remarque préliminaire:

.


\begin{lemme}\label{carrenefmodif>0} Soit $M$ Kähler compacte et $\beta\in H^{1,1}(M,\mathbb R)$ une classe nef modifiée, alors $q(\beta,\beta)\leq 0$.
\end{lemme}

\begin{proof} Par définition, il existe pour  $\varepsilon>0$ arbitrairement petit un $(1,1)$ courant positif fermé $T_\varepsilon$ lisse en dehors d'un ensemble analytique de codimension 2 et tel que $\{T_\varepsilon\}\in \beta +\varepsilon\{\theta\}$. Suivant \cite{de}, il en résulte que $T_\varepsilon\wedge T_\varepsilon$ est bien défini en tant que $(2,2)$ courant {\it positif} fermé et que ${\{T_\varepsilon\}}^2=\{T_\varepsilon\wedge T_\varepsilon\}$. Le lemme s'obtient alors par passage à la limite.
\end{proof}
\vskip 20 pt
{\bf Preuve de la proposition \ref{Z(alpha)nef}} 
\vskip 5 pt
i): soit 

$$\sum\nu(T,D)[D] + R$$
la décomposition de {\it Siu} du courant $\Xi$. Puisque $\Xi$ est $\mathcal F$ invariant, on obtient que chaque diviseur premier $D$ tel que $\nu (T,D)\not=0$ est une s\'eparatrice du feuilletage (de façon équivalente, le courant d'intégration $[D]$ est $\mathcal F$ invariant).

Soit $$\alpha=\{N(\alpha)\}+Z(\alpha)$$ la décomposition divisorielle; compte tenu de la propriété iii) de la proposition \ref{propriétészar}, chaque $D_i$ est nécessairement $\mathcal F$ invariant. Par suite, le courant positif $\Xi-\sum \lambda_i [D_i]$ représentant $Z(\alpha)$ est également $\mathcal F$ invariant.

\par ii) : d'après i), le lemme \ref{carrenefmodif>0} et les formules (\ref{produit coh}) et (\ref{nulXiwedge eta}) ,la forme symétrique $q$ est semi-négative en restriction au $\mathbb R$ espace vectoriel $V$ engendré par $Z(\alpha)$ et $\{\eta_T\}$;  suivant Hodge-Riemann, on a donc nécessairement dim $V=1$.

\par iii): conséquence  du point i), de (\ref{nuletawedge eta}) et le fait déjà observé que $\{\eta_T\}$ est {\it nef}.

\par iv): conséquence de (\ref{produit coh}) et (\ref{nulXiwedge eta}).

\par v):  
soit $W$ (resp. $W^+$) l'espace vectoriel (resp. le cône) engendré par les $\{D_i\}$. En exploitant le fait déjà établi que ${\{\eta_T\}}^2=0$, que  $\{\eta_T\}\in W^\perp$ pour la forme bilinéaire $q$ et la propriété {\it iv)} de la proposition \ref{propriétészar}, on obtient par le théorème de la signature que $q(\alpha_1,\alpha_1)>0$ si $\alpha_1\in W^+\setminus\{0\}$ . Par ailleurs, on a
$q(\{D_i\},\{D_j\})<0$ si $i\not= j$, en conséquence de quoi on a $q(\alpha_2,\alpha_2)>0$ si $\alpha_2\in W\setminus\{0\}$. En particulier, $\alpha_2$ n'est pas colinéaire à $\{\eta_T\}$.\\ 

On peut reformuler ce résultat en disant que la projection de $W$ sur $\mathcal P$ parallèlement à la droite dirigée par $\{\eta_T\}$ n'est pas dégénérée, ce qui nous conduit au

\begin{lemme}\label{carrec1not=0} La forme $q$ est définie positive sur $W$.  En particulier, $\alpha^2$ n'est pas nulle (et plus précisément $q(\alpha,\alpha)>0$) lorsque $N(\alpha)$ n'est pas triviale.

\end{lemme}

\par vi): Supposons maintenant que $M$ est projective; on peut choisir la forme de Kähler $\theta$ de telle sorte que la forme $q$ soit rationelle, i.e: à valeurs rationnelles sur $E_\QQ=NS(M)\otimes\QQ$ où l'on note $NS(M)$ le groupe de Néron-Séveri de la variété. Soit $V_\QQ$ l'espace vectoriel engendré par les $\{D_i\}$.  Compte tenu des propriétés ii)...v) et du lemme précédent, $\{N(\alpha)\}$ doit coïncider avec la projection de $\alpha$ parallèllement à ${V_\QQ}^\perp$, en particulier $\lambda_i\in\QQ$ pour tout $i$.
\par vii):résulte directement de (\ref{integ}) lorsque $\alpha= c_1({N_{\mathcal F}}^*)$.   

Plus généralement, soit $S$ un $(1,1)$ courant positif fermé tel que $\{S\}=\alpha$. D'après (\ref{deflambdai}) et le point i) de la proposition, le courant $S_0:= S-N(\alpha)$ est positif et il suffit de montrer que $S_0$ est $\mathcal F$ invariant. Localement, on peut écrire $$S_0= \frac{i}{\pi}\partial\overline{\partial}\varphi$$
avec $\varphi$ {\it psh}. Soit $N:=N(c_1({N_{\mathcal F}}^*))$ et $Z:=Z(c_1({N_{\mathcal F}}^*))$.
Traitons d'abord le cas $Z=0$;  en utilisant (\ref{produit coh}) et (\ref{nulXiwedge eta}), on constate que le $(2,2)$ courant positif fermé $\eta_T\wedge S_0$ est identiquement nul.
 Par ailleurs, on a  localement, $\eta_T=i{|f|}^\delta\omega\wedge\overline{\omega}$
où $f$ est une équation locale de l'hypersurface $H=\vert N\vert$, $\delta$ un réel positif et $\omega$ un générateur de ${N_{\mathcal F}}^*$. Ceci montre que $S_0$ est $\mathcal F$ invariant en dehors de $H$, donc sur $M$ puisqu'on a vu que $H$ est invariante par le feuilletage.

Supposons maintenant $Z\not=0$;   il suffit de montrer (i) prop 2.8)  que le$(1,1)$ courant positif $S_0-[N(\alpha)]$(qui représente donc $Z(\alpha)$) est invariant par le feuilletage et on peut donc supposer $Z(\alpha)\not =0$. En utilisant le point ii) de la même proposition, il existe alors $\lambda >0$ tel que $Z=\lambda Z(\alpha)$. Le courant  $[N]+ \lambda (S_0-[N(\alpha)])$ représente donc  $c_1({N_{\mathcal F} }^*)$, d'où l'invariance de $S_0$.

viii): posons $\alpha=\sum_{i=1}^r\{H_i\}$ et supposons la famille exceptionnelle, ce qui revient à dire que $\alpha=\{N(\alpha)\}$. La négativité de $(m_{ij})$ résulte alors de (\ref{produit coh}), (\ref{nulXiwedge eta}), (\ref{nuletawedge eta}) combinés avec le point v) de la proposition précédente et le théorème de Hodge-Riemann.

Réciproquement, supposons $(m_{ij})$ définie négative; la partie positive de $\alpha$ est alors de la forme $Z(\alpha\}= \sum_{i=1}^r\mu_i\{H_i\}$ avec $0\leq\mu_i\leq 1$. D'après iii) de la proposition \ref{Z(alpha)nef} et compte tenu des hypothèses faites sur $(m_{ij}$ ceci n'est possible que si
$\mu_i=0$ pour tout $i$ et implique bien que la famille est exceptionnelle.  \qed
\begin{rem}
Nous pensons que la propriété iv) persiste sans hypothèses de projectivité mais nous ne connaissons pas de preuves.

Notons qu'en dehors du contexte des  feuilletages, la rationalité de la décomposition n'est pas toujours vraie (\cite{cu}).

Il est également naturel de conjecturer que le support de la partie négative $N(c_1({N_\mathcal F}^*)$ est une hypersurface contractible.

\end{rem}

\begin{prop} Soit $\alpha =c_1({N_{\mathcal F}}^*)$. Supposons que le feuilletage admette en tout point singulier  un germe réduit d'intégrale première holomorphe; alors $N(\alpha)$ est triviale.

 \end{prop}

\begin{proof} En effet, dans cette situation, les générateurs locaux du conormal peuvent être donnés par des formes holomorphes {\it fermées} qui se recollent donc suivant un cocycle multiplicatif $\{g_{UV}\}$ qui représente le fibré $N_\mathcal F$ et qui est {\it localement constant} sur les feuilles. D'un point de vue différentiable, la cohomologie de ce fibré est triviale à partir du rang 1; on peut donc trouver, pour tout couple d'ouverts $(U,V)$ du recouvrement , $h_U\in{\mathcal C}^\infty (U), h_V\in{\mathcal C}^\infty (V)$ tel que sur $U\cap V$, on ait

$$\frac{dg_{UV}}{g_{UV}}=h_U\omega_U-h_V\omega_V$$

\noindent avec $\omega_U$, $\omega_V$ des formes holomorphes {\it fermées}  définissant le feuilletage sur les ouverts considérés; sur chaque intersection, on a donc 
$$dh_U\wedge\omega_U=dh_V\wedge\omega_V=\Omega.$$
 La  $2$ forme différentielle $\Omega$ représente (au sens de de Rham et à un facteur près) la classe de Chern $\alpha$. Par construction, $\Omega\wedge \Omega=0$ et par suite $\alpha^2=0$, on conclut par le lemme \ref{carrec1not=0}.

\end{proof}

\section{Construction d'une métrique ``canonique'' sur l'espace des feuilles}\label{metsurespfeuilles}

 Soit $$\alpha=\{N\}+Z$$ ($N=\sum_{i=1}^p\lambda_iD_i=N(\alpha), Z=Z(\alpha)$) la décomposition de Zariski de la classe $$\alpha=c_1({N_\mathcal F}^*).$$ 

 On a précédemment établi que 

-$Z$ est en fait {\it nef},

-tout courant $(1,1)$ positif représentant $N, Z$ ou $\alpha$ est $\mathcal F$ invariant,

-à tout $(1,1)$ courant positif $T$ tel que $\{T\}=\alpha$ est canoniquement associée (modulo multiplication par un scalaire $>0$) une $(1,1)$ forme positive fermée $\eta_T$  (formule (\ref{eta}) et de plus, $Z$ est proportionnelle à $\{\eta_T\}$. On adoptera donc par la suite la normalisation

$$\eta_T=Z$$
lorsque $Z\not=0.$

\begin   {THM}  \label{thme princ} Il existe  un courant positif $T$ de bidegr\'e $(1,1)$ invariant par $\mathcal F$ tel que $\{T\}=\alpha$ et 
 \begin{equation}
  T=[N]+\varepsilon\eta_T\label{equationcourbure}
 \end{equation}
où $\varepsilon=0$ si $Z=0$ et  $\varepsilon=1$ sinon.

Ce courant $T$ est par ailleurs unique.
\end   {THM}

\begin{rem}\label{integprem}

En dehors  de l'hypersurface $H=$ supp $N$, l'égalité (\ref{equationcourbure}) indique, par des résultats classiques d'ellipticité,  que $T$ est lisse. Il est donc en particulier à singularités minimales.

Par ailleurs, en un point {\it régulier} $m$ du feuilletage  appartenant à $M\setminus H$, les feuilles de $\mathcal F$  sont données par les niveaux d'une submersion $z$ et il existe donc une fonction lisse sous-harmonique  $\varphi$ ne dépendant que de $z$ telle qu'au voisinage de $m$,

$$T=\frac{i}{\pi}\partial\overline{\partial}\varphi\ et\ \frac{{\partial}^2\varphi}{\partial z\partial\overline{z}}=\varepsilon e^{2\varphi}$$

ce qui munit l'epace local des feuilles d'une métrique de courbure nulle ou $-1$ (modulo normalisation) suivant que $\varepsilon=0$ ou $1$.

Dans la terminologie des feuilletages, $\mathcal F$ est donc soit {\bf transversalement euclidien}, soit {\bf transversalement hyperbolique} sur $M\setminus(H\cup Sing(\mathcal F))$ (cet aspect sera détaillé dans la section \ref{unifo}). 
\end{rem}

\begin{proof} Elle est immédiate si $Z(\alpha)=0$; nous supposerons donc dorénavant cette classe {\it non triviale}.

 Commençons par établir l'unicité. Soient $$T_1=\frac{i}{\pi}\partial\overline{\partial}\varphi_1\ et\ T_2=\frac{i}{\pi}\partial\overline{\partial}\varphi_2$$ deux courants  satisfaisant aux conclusions du théorème \ref{thme princ}. On peut choisir les potentiels locaux {\it psh} $\varphi_i$ de telle sorte que 

$$ u=\varphi_2-\varphi_1$$ soit une fonction bien d\'efinie sur la vari\'et\'e $M$ et 

$$\eta_{T_2}=\frac{i}{\pi}e^{2\varphi_2}\omega\wedge\overline{\omega}=e^{2u}\eta_{T_1}.$$

Pla\c cons nous {\it au voisinage} d'un point $m$ de $M$. 

Soient $\{f_1=0\},....,{f_p=0}\}$ des \'equations r\'eduites respectives de $D_1,...,D_p$.

Les fonctions {\it psh} $\psi_k=\varphi_k-\sum_{i=1}^p\lambda_i\log (|f_i|),\ k=1,2$ satisfont l'EDP

      $$\partial\overline{\partial}\psi_k=\prod_{i=1}^p{|f_i|}^{2\lambda_i}e^{2\psi_k}\omega\wedge\overline{\omega}$$
et sont donc en particulier continues (en fait ${\mathcal C}^{2-\varepsilon}$). Par suite, $u=\psi_2-\psi_1$ est continue.

Elle satisfait de plus

$$\partial\overline{\partial}u=(e^{2{(\varphi_1+u})}-e^{2{\varphi_1}})\omega\wedge\overline{\omega}$$

En particulier, $u$ est {\it psh} sur l'ouvert $\{u>0\}$ qui est donc n\'ecessairement vide par le principe du maximum. Comme $e^{\varphi_1}\omega\wedge\overline{\omega}$ et $e^{\varphi_2}\omega\wedge\overline{\omega}$ sont cohomologues, on en déduit que $u=0$.

Reste à prouver l'existence; à cet effet, introduisons l'ensemble $\mathcal C$ formé des $(1,1)$ courants positifs ferm\'es $T$ tels que $\{T\}=Z(\alpha)$; c'est un sous-ensemble convexe et faiblement compact de l'espace des courants. Rappelons que ses éléments sont $\mathcal F$ invariants.  Considérons également le courant représentant la partie négative $N$:

     $$ S=\sum_{i=1}^p\lambda_i[D_i] $$

Suivant la formule (\ref{eta})et d'après les résultats de la section 2,  on hérite d'une application  $\beta:\mathcal C\rightarrow \mathcal C$ définie comme suit:

pour tout $T\in\mathcal C$,

$$\beta (T)=\eta_{T+S},$$  ce qui a effectivement un sens après normalisation $\{\eta_{T+S}\}=Z(\alpha)$.

Le théoréme \ref{thme princ} et alors une simple conséquence du théorème du {\it point fixe de Tychonoff}, compte-tenu du 

\begin{lemme}
  L'application $\beta$ ci-dessus est continue (pour la topologie faible).
\end{lemme}

Prouvons ce lemme.
Soit $T\in\mathcal C$ et $(T_n)$ une suite de courants de $\mathcal C$ convergeant faiblement vers $T$

Par compacit\'e de $\mathcal C$, on peut supposer que la suite $(\beta (T_n))$ converge; il s'agit alors de montrer que sa limite est {\it précisément $\beta (T)$}.  

Considérons un recouvrement $(U_j)_{j\in J}$ de $M$ par des ouverts de Stein contractiles $U_j$. Sur chaque $U_j$, tout courant fermé positif de bidegré $(1,1)$ admet donc un potentiel {\it psh}. Quitte à raffiner ce recouvrement, on peut exhiber sur chaque $U_j$ une collection de fonctions $f_{j,i}\in\mathcal O (U_j)$ définissant localement $D_i,i=1,...,p$ par les équations (réduites) respectives $f_{j,i}=0$ ainsi qu'une forme différentielle holomorphe $\omega_j$  (sans diviseur de zéros) définissant  le feuilletage. Ces formes se recollent sur les intersections $U_j\cap U_k$ suivant le cocycle multiplicatif $(g_{jk})$:

$$\omega_j=g_{jk}\omega_k$$
qui représente le fibré normal $N_\mathcal F$ du feuilletage.

 Pour tout $n$,  il existe sur chaque ouvert $U_j$ une fonction psh    $\varphi_{n,j}$ telle que 

$$\frac{i}{\pi}\partial\overline{\partial}\varphi_{n,j}=T_n$$
dont l'unicit\'e est assur\'ee d\`es qu'on se prescrit les conditions suivantes:

\noindent pour tout $n$,
$$ sup_j\ \varphi_{n,j}=0\ ,\ \varphi_{n,j}-\varphi_{n,k}=-\log |g_{jk}|-\sum_i\lambda_i\log |f_{j,i}|+\sum_i\lambda_l\log |f_{k,i}|$$
Il existe par conséquent une suite de réels $(c_n)$ telle que pour tout $n$, on ait en restriction à l'ouvert $U_j$:

   $$\beta (T_n)=\eta_{T_n+S}=\frac{i}{\pi}e^{2\varphi_{n,j}+2\sum_i\lambda_i\log |f_{j,i}|+c_n}\omega_j\wedge\overline{\omega_j}$$

Quitte  \`a extraire une sous-suite, on peut supposer (\cite{ho}) que pour tout $j$, $(\varphi_{j,n})_{n\in\mathbb N}$ converge dans $L^1_{loc}(U_j)$ vers une fonction plurisousharmonique $f_j$  telle que sur les intersections $U_j\cap U_k$, on ait:

$$f_j-f_k=-\log |g_{jk}|-\sum_i\lambda_i\log |f_{j,i}|+\sum_i\lambda_l\log |f_{k,i}|$$

 Par construction, la collection des $\frac{i}{\pi}\partial\overline{\partial }f_i$ donnent lieu, par recollement, à un courant positif qui n'est rien d'autre que $T$.

Pour une constante réelle $c$ {\it ad hoc}, on obtient donc qu'en restriction \'a une carte locale $U_i$,

$$\beta(T)=\frac{i}{\pi}e^{2f_{j}+2\sum_i\lambda_i\log |f_{j,i}|+c}\omega_j\wedge\overline{\omega_j}$$

Posons $u_n=e^{c-c_n}\eta_{T_n+S}$; sur $U_j$, on a

    $$u_n-\beta (T)=\frac{i}{\pi}e^c\prod_i |f_{j,i}|^{2\lambda_i}e^{2({\varphi_{n,j}-f_j})}\omega_j\wedge\overline{\omega_j}$$

Fixons une forme test $\xi\in\Omega_c^{n-1,n-1}(U_j)$. Compte-tenu du fait que $\varphi_{n,j}$ et donc $f_j$ ne prend pas de valeurs $>0$,on déduit du théorème des accroissements finis qu'il existe une constante $D_j$ telle que pour tout $n$, on ait

$$|\langle u_n-\beta (T),\xi\rangle|\leq D\int_K|\varphi_{n,j}-f_j|dm$$
où $K$ désigne le support de $\xi$ et $m$ la mesure de lebesgue sur $U_j$. Au moyen d'une partition de l'unité, on en déduit que $u_n$ converge vers $\beta (T)$ sur la variété $M$. Puisque $\beta (T_n)$ est par ailleurs convergente, la suite $(c_n)$ admet une limite, nécessairement égale à $c$ puisque $\{\beta (T)\}=\{\beta (T_n)\}$. On a donc bien que $\beta (T_n)$ converge vers $\beta (T)$ 

\end{proof}

L'objet de la section qui suit est de décrire la nature des singularités pouvant apparaître dans la classe de feuilletages que nous étudions. Il sagit d'une étude purement locale qui peut éventuellement présenter un intérêt par elle-même.

\section{Analyse locale des singularités}

Les objets considérés ici seront de nature locale et il sera souvent commode d'adopter le langage des germes, sachant que , par abus de langage, la confusion sera souvent faite entre ``germe'' et ``représentant d'un germe''.
\medskip

On se donne sur  $({\CC}^n,0)$, un germe de  feuilletage d\'efini par une forme holomorphe int\'egrable $\omega$, ainsi qu'une fonction {\it psh} $\varphi$ telle que 

     $$\eta=ie^{2\varphi}\omega\wedge\overline{\omega}$$
soit ferm\'ee (au sens des courants) et en particulier $\mathcal F$ invariante.\\

\begin{THM} \label{structlocale} Soit $\mathcal F$ un germe de feuilletage holomorphe de codimension 1 v\'erifiant les hypoth\`eses ci-dessus. Alors $\mathcal F$ admet une int\'egrale premi\`ere {\bf élémentaire}, c'est-à-dire du type $f=\prod_{i=1}^mf_i^{\gamma_i}$ o\`u les $\gamma_i$ sont des réels positifs et $f_i\in{\mathcal O}_n$.
\end{THM}


D'après \cite{cema}, il suffit en fait d'\'etablir ce r\'esultat en restriction \`a un 2-plan g\'en\'erique; on peut donc supposer que $n=2$.

On constate par ailleurs que $\mathcal F$ ne peut être dicritique (sinon, cela cr\'eerait un ph\'enom\`ene de ``concentration de masse'' et forcerait le courant $\eta$ à admettre un nombre de Lelong non nul à l'origine).\\

 La preuve du théorème 2 s'articule alors comme suit Elle s'articule comme suit; consid\'erons le morphisme $\pi$ de r\'eduction des singularit\'es de $\mathcal F$. On sait identifier les feuilletages à singularit\'es r\'eduites susceptibles d'admettre un courant invariant \cite{br}. Dans le cas pr\'esent, o\'u le courant $T$ a une forme tr\'es particuli\'ere, cette liste peut être affin\'ee. En étudiant les germes de courants d\'ependant d'une variable complexe de la forme $ie^{2\varphi} dz\wedge d\overline{z}$ et invariants par un germe de Diff$(\mathbb C,0)$, on peut par ailleurs analyser la structure des groupes d'holonomie projective.
On recolle ensuite ces informations locales et semi-locales (dans l'esprit de ce qui est fait dans \cite{pa}).\\

\begin{lemme} \label{TinvparH}
  Soit $T$ un germe de courant dans $(\mathbb C,0)$ de la forme 

$$T=ie^{2\varphi} dz\wedge d\overline{z}$$
o\`u $\varphi$ est psh; supposons que $T$ soit invariant par l'action d'un sous groupe $H$ finiment engendré de $Diff({\CC},0)$. Soit $h\in H$; pour tout voisinage $V$ de $0$, il existe un ouvert $U\subset V$ contenant $0$ tel que $h(U)=U$. 

\end{lemme}

\begin{cor}
Sous les hypothèses du lemme \ref{TinvparH}, $H$ est analytiquement conjugué à un sous-groupes du groupe des rotations 

$$R=\{z\rightarrow e^{2i\pi\lambda}z\}$$

i.e, il existe $\varphi\in Diff({\CC},0)$ tel que

$$\varphi^{-1}H\varphi\subset R$$
\end{cor}

En effet, par le théorème du domaine invariant chaque élément $h$ de $H$ est analytiquement conjugué à une rotation. En particulier, $H$ ne comporte pas de germe non trivial tangent à l'identité et est donc abélien. Il est alors bien connu (cf par exemple \cite{lo}) que tout sous-groupe abélien de type fini de $Diff({\CC},0)$ dont chaque élément est analytiquement linéarisable est lui-même analytiquement linéarisable.\\

 {\it Preuve du lemme \ref{TinvparH}}

Soit $h\in H$ tel que $h^*T=T$. Supposons par l'absurde qu'il n'existe aucun ouvert $U\ni 0$ tel que $h(U)=U$.

Soit ${\mathbb D}_r\subset\CC$ le disque ouvert $\{|z|<r\}$. Choisissons $0<r_1<r_2$  suffisamment petits de telle sorte que $T=ie^{2\varphi} dz\wedge d\overline{z}$ soit d\'efini sur ${\mathbb D}_{r_2}$, que $h$ soit univalent sur ${\mathbb D}_{r_1}$ et $h({\mathbb D}_{r_1})\subset {\mathbb D}_{r_2}$.

 Soit $\gamma\subset {\mathbb D}_{r_1}$ une courbe rectifiable. Sa longueur, suivant la forme m\'etrique $g=e^{\varphi}|dz|$ est donn\'ee suivant la formule

$$l_g(\gamma)=\int_\gamma e^{\varphi (z)}d\mathcal H$$

o\`u $\mathcal H$ d\'esigne la mesure de Haussdorff de dimension 1; pour $a\in\CC$, notons $\gamma_a$ le segment $[0,a]$. La non-existence de domaine invariant par $h$ implique qu'il existe une suite $(a_n)$ de points de ${\mathcal D}_{r_1}\setminus K_{r_1}$ convergeant vers l'origine ainsi qu'une suite d'entiers relatifs $k_n$ telle que 

$$h^{k_n} (\gamma_{a_n})\in {\mathbb D}_{r_2}\ ,h^{k_n} (\gamma_{a_n})\cap ( {\mathbb D}_{r_2}\setminus {\mathbb D}_{r_1})\not= \emptyset $$ 

Compte tenu des propri\'et\'e d'invariance de $T'$, on a par ailleurs

$$l_g(h^{k_n} (\gamma_{a_n}))=l_g(\gamma_{a_n}).$$

Si l'on reprend le raisonnement men\'e dans \cite{brsurf}, on constate que qu'il existe un r\'eel strictement positif $M$ tel que pour tout $n$, on ait

$$l_g(h^{k_n} (\gamma_{a_n}))>M$$
alors que $l_g(\gamma_{a_n})$ converge vers $0$, ce qui m\`ene \'evidemment à une contradiction.\qed\\ \\

 {\it Preuve du théorème \ref{structlocale}} 

Consid\'erons l'arbre de r\'eduction $\mathcal A$ du feuilletage ${\mathcal F}_\omega$ donn\'e sur un voisinage $V$ de l'origine dans ${\mathbb C}^2$ par $\omega=0$.

On notera 

$$\rho:U\rightarrow V$$ le morphisme de r\'eduction, d\'efini sur un voisinage $U$ de $\mathcal A$ et compos\'e d'une succession d'\'eclatements ponctuels.\\

Par hypoth\`ese, le feuilletage r\'eduit  $\tilde {\mathcal F}=\rho^*\mathcal F$ admet un $(1,1)$ courant positif ferm\'e invariant par holonomie

$$\tilde T=\pi^*T:=ie^{2\phi\circ\rho}\rho^*\omega\wedge\rho^*\overline{\omega}$$


Par construction, $\tilde T$ est localement de la forme $ie^{2\tilde\varphi}\tilde\omega\wedge\overline{\tilde\omega}$ o\`u $\tilde\varphi$ est {\it psh} et $\tilde\omega$ un générateur du feuilletage saturé  $\tilde {\mathcal F}$.

Soit $S$ l'ensemble des s\'eparatices de $\mathcal F_\omega$. Puisqu'on est dans une situation {\it non dicritique}, $S$ est constitu\'ee d'une union finie de $p$ courbes irr\'eductibles.

Désignons par  $\tilde S$ la transform\'ee sricte de $S$ par $\pi$. D'apr\`es l'analyse pr\'ec\'edente, les singularit\'es qui apparaissent apr\`es r\'eduction sont du type Siegel lin\'earisable, c'est \`a dire sont donn\'es par des formes diff\'erentielles qui s'\'ecrivent à conjugaison  {\it analytique} pr\`es

   $$zdw+\lambda wdz,\ \lambda >0$$
et qui admettent par cons\'equent des intégrales premières du type $zw^\lambda$.

L'ensemble des singularit\'es Sing $\tilde{\mathcal F}$ du feuilletage r\'eduit  $\tilde{\mathcal F}$ est alors exactement localis\'e aux intersections des composantes irr\'eductibles de $\pi^{-1}(S)$.

Nous reprenons  pour ce qui suit la terminologie et les notations adopt\'ees dans \cite{pa}. Soit $Z\subset \tilde S$  une zone holomorphe maximale et $f_Z$ une transversale {\it holomorphe} d\'efinie sur un voisinage $U_Z$ de $Z$.

Supposons $Z\not= \tilde S$ et consid\'erons une composante adjacente $C\subset \mathcal A$. Soit $C^*=C\setminus Sing \ {\tilde {\mathcal F}}$.

Choisissons un point $m\in C^*$ ainsi qu'une transversale locale $f_m$ en $m$ au feuilletage ${\tilde{\mathcal F}}$. Soit $T_m= (\CC,0)$ l'image de cette transversale locale. Au voisinage de $m$, on a $f_Z=l(f_m)$ o\`u $l\in {\mathcal O}_1$. Suivant \cite{pa}, on introduit le {\it groupe d'invariance de $f_Z$ par rapport \`a $f_m$}:

$$Inv(f_Z,f_m)=\{g\in Diff(T_m)|l\circ g=l\}$$
qu'on enrichit par le groupe  $G_m\in Diff(T_m)$ d'holonomie projective de $C^*$ afin d'obtenir le groupe $\overline{G_m}$ engendr\'e par $ Inv(f_Z,f_m)\cup G_m$.

D'apr\`es ce qui pr\'ec\`ede, $G$ est analytiquement conjugu\'e \`a un sous-groupe dense du groupe des rotations et ceci entraîne, d'apr\`es \cite{pa} que ${\tilde{\mathcal F}}$ admet au voisinage de $Z\cup C$ une int\'egrale premi\`ere (une transversale dans le langage de \cite{pa}) logarithmique; c'est-\`a-dire une intégrale premi\`ere multiforme $f$ telle que 
$df/f$ soit une forme ferm\'ee \`a p\^oles simples d\'efinissant le feuilletage sur la zone consid\'er\'ee (et bien d\'efinie modulo multilication par une constante). 

Apr\`es normalisation et compte-tenu de la description des singularit\'es locales, on peut de plus affirmer que les r\'esidus de cette forme ferm\'ee sont des r\'eels positifs.

Supposons maintenant qu'il existe une zone logarithmique non holomorphe; des raisonnements similaires (voir \cite{pa}) pour les d\'etails)  montrent qu'elle s'étend en une zone logarithmique sur les composantes exceptionnelles adjacentes. Ceci achève la preuve du théorème \ref{structlocale}.
Nous revenons maintenant au cadre global étudié dans les trois premières sections.\qed

\section{Structure du feuilletage au voisinage d'une hypersurface invariante}

Nous réadoptons les notations de la section \ref{metsurespfeuilles}: $\alpha:=c_1({N_\mathcal F}^*)$ est supposée pseudo-effective et admet $\alpha=\{N\}+Z$ comme décomposition de Zariski.

Il pourra être utile d'utiliser la caractérisation des singularités locales pour donner une ``description'' du feuilletage $\mathcal F$ près d'une hypersurface invariante et plus particulièrement au voisinage de $H=\vert N\vert$.

Soit $p$ un point de $M$ tel que $p\notin H$. Quand $p$ est de plus un point régulier du feuilletage, le conormal est engendré par $dz$ dans une coordonnée $z$ adéquate.

La forme positive $\eta_T$ mentionnée dans le théorème \ref{thme princ}  s'explicite alors comme suit:

$$\eta_T=\frac{i}{\pi} e^{2\varphi (z)}dz\wedge d{\overline z}$$
avec $\Delta \varphi:=\frac{\partial^2\varphi}{\partial z\partial{\overline z}}=\varepsilon e^{2\varphi}$, avec $\varepsilon=0$ (cas {\it transversalement euclidien}) ou $\varepsilon=1$ ({\it cas transversalement hyperbolique}).

Rappelons comment sont construites ces structures transverses.

 Soit $d^0$ la métrique euclidienne standard $idu\wedge d\overline{u}$ sur la droite complexe $U^0=\CC$ et $d^1=\frac{i}{\pi}\frac{du\wedge d\overline{u}}{{(1-{|u|}^2)}^2}$ la métrique de Poincaré (convenablement normalisée) sur le disque $U^1=\DD=\{\vert u\vert <1\}$ (plus exactement leurs formes d'aire respectives).

Soit $I^\varepsilon$ le groupe des isométries conformes de $U^\varepsilon, \varepsilon=0,1.$ Il existe alors au voisinage de chaque point $p\notin H$ un germe d'intégrale première  $f$ à valeur dans $U^\varepsilon$ telle que $\eta_T=f^*d^\varepsilon$ et $f$ est uniquement définie modulo l'action à gauche de $I^\varepsilon$. C'est un fait classique lorsque $p$ est régulier (auquel cas $f$ est une submersion) qui s'étend sans difficultés à $p\in\ Sing\ \mathcal F\setminus H$ en utilisant que 
$V\setminus Sing\ \mathcal F$ est simplement connexe pour un voisinage ouvert $V$ convenablement choisi et arbirairement petit de $p$ ainsi que le théorème de prolongement d'Hartogs. Notons que le lieu critique de $f$ coïncide localement avec $Sing\ \mathcal F$ et qu'il ne s'agit donc plus nécessairement d'une submersion. 

\begin{definition}
Le faisceau défini sur $M\setminus H$ et déterminé par la collection de ces  germes $f$ (lorsque $p$ varie) est appelé {\bf faisceau des intégrales premières admissibles} et sera noté ${\mathcal I}^\varepsilon$.
\end{definition}

Il sera également commode d'y adjoindre le faisceau suivant.
\begin{definition}
Le faisceau défini sur $M\setminus H$ et déterminé par la collection des dérivées logarithmiques $\frac{df}{f},f\in{\mathcal I}^\varepsilon$  est appelé {\bf faisceau des dérivées logarithmiques admissibles} et sera noté ${\mathcal I}_{d\log}^\varepsilon$
\end{definition}

\bigskip
Rappelons (théorème \ref{structlocale}) que le feuilletage possède au voisinage de tout point $p$ de $M$ une intégrale première ``élémentaire'' du type $$F=\prod_{i=1}^mf_i^{\gamma_i}$$ et admet par conséquent comme seule séparatrice locale le germe d'hypersurface $X_p=\{\prod_{i=1}^mf_i=0\}$. Cette situation inclut évidemment le cas où $\mathcal F$ est régulier en $p$!

\begin{lemme}\label{logloc}Sur un voisinage suffisamment petit $V_p$ de $p$,le feuilletage ${\mathcal F}_{V_p\setminus X}$ est définie par une section de ${\mathcal I}_{d\log}^\varepsilon$ qui s'étend en une forme logarithmique $\eta_p$ sur $V_p$ à pôles exactement dans sur $X$. Cette forme $\eta_p$ est par ailleurs unique.
 \end{lemme}

\begin{proof}
 Traitons en premier lieu le cas où $\mathcal F$ est {\bf régulier }en $p$. Si  $p$ appartienne à une composante $H_{i_0}$ de $H$, on convient de poser $\lambda=sup_{\varepsilon>0}\nu(T_{min,\varepsilon}, H_{i_0 })$ (on réadopte ici les notations du théorème \ref{deczar}) et $\lambda=0$ sinon.  Soit $(T_p,p)\simeq (\DD,0)$ un germe de transversale à $\mathcal F$ en $p$.

La restriction de $\eta_T$ à $T_p$ est alors de la forme $\tilde{\eta_T}=\frac{i}{\pi} e^{2\varphi (z)}dz\wedge d{\overline z}$ où $\varphi$ est une fonction sous harmonique sur $T_p$, lisse en dehors de $p$ et vérifiant l'EDP
\begin{equation}\label{massedirac}
\Delta \varphi:=\frac{\partial^2\varphi}{\partial z\partial{\overline z}}=\varepsilon e^{2\varphi}+\lambda\delta_p
\end{equation}

où $\delta_p$ est la masse de Dirac en $p$.

Sur tout ouvert sectoriel de $T_p$ de la forme  $S=\{z\not=0;\alpha<\arg z<\alpha + 2\pi\}$, il existe par simple connexité une section $f$ de ${\mathcal I}^\varepsilon$. 

Soit $\gamma: [0,1]\rightarrow T_p$ une courbe rectifiable telle que $\gamma (1)=p$ et $\gamma (t)\in S, t\not=1$. Puisque $\tilde{\eta_T}(=f^*d^\varepsilon$ sur $S$) est à coefficients bornés sur $T_p$, on obtient que $f\circ\gamma (t)$ admet une limite $l$ dans $U^\varepsilon$ quand $t\rightarrow 1$, laquelle est indépendante du chemin $\gamma$ considéré.

Soit $\tau_{c} (f), \tau_c\in I^\varepsilon$ la section de ${\mathcal I}^\varepsilon$ définie sur $S$ et obtenue par prolongement de $f$ le long d'un lacet $c:[0,1]\rightarrow T_p\setminus\{p\}$. On a visiblement $\tau_c (l)=l$, de sorte qu'on peut se ramener au cas ou $l=0$ et $\tau_c$ est donc une rotation fixant $0$. On obtient alors facilement que $f(z)=z^\mu h(z)$ où $\mu$ est un certain réel positif tel que $e^{2i\pi \mu}$ soit l'angle de la rotation et $h\in {\mathcal O}^*$. Quitte à se placer dans une coordonnée holomorphe appropriée, on peut donc supposer que $f(z)=z^\mu$; on conclut finalement que $\tilde{\eta_T}=f^*d^\varepsilon={\mu}^2\frac{i}{\pi}{\vert z\vert}^{2\mu -2} {dz\wedge d\overline{z}}$ si $\varepsilon =0$

 \hskip 5 pt $={\mu}^2\frac{i}{\pi}{\vert z\vert}^{2\mu -2}\frac{dz\wedge d\overline{z}}{{(1-{|z|}^{2\mu})}^2}$ si $\varepsilon=1$.
 
 (ces égalité, vraies sur $T_p\setminus \{p\}$ s'étendent en fait sur $T_p$ car les mesures impliquées dans les membres de gauche et droite n'affectent pas de masse au point $p$).
 
 En particulier, on a $\mu>1$ et $\varphi (z)= (\mu -1)\log\vert z\vert + O(1)$ ce qui entraîne finalement d'après (\ref{massedirac}) que $\mu=\lambda +1$.
 . La forme $\eta_p$ est alors l'extension méromorphe de $\frac{df}{f}$ sur $T_p$ (et donc sur un voisinage de $p$ par flow-box).
 
 Il reste à considérer le cas où $p\in H$ est également une singularité du feuilletage. On se ramène à la situation précédente; pour ce faire, choisissons $p'$ régulier, $p'\in X_p$, ainsi qu'une transversale $T_{p'}$, un secteur $S'\in T_{p'}$, une section $f\in {\mathcal I}^\varepsilon (S')$ et une forme $\eta_{p'}$ définies comme ci-dessus. Pour $S'$ convenablement choisi, la section locale $f$ se factorise à travers une détermination $\overline{F}$ de l'intégrale première élémentaire $F$ suivant
 
 $$f(z)={(\psi ({\overline{F}^\alpha(z)})}^\beta$$ 
 
 avec $\alpha,\beta>0$ et $\psi$ un difféomorphisme local de $(\CC,0)$. Soit $\delta>0$ suffisamment petit et $V_\delta\subset V_p$ la composante connexe de $\vert F\vert<\delta$ contenant $p$ (et donc $p'$). En exploitant maintenant à nouveau que $Sing \mathcal F$ est de codimension $2$ et ce qui précède, on observe que le prolongement analytique de $\eta_{p'}$ dans  $V_\delta$ induit la forme $\eta_p$ requise au voisinage de $p$.
\end{proof}

Soit $K$ une hypersurface invariante par le feuilletage et $K_1,..., K_r$ ses composantes irréductibles.

Posons $\nu_i=sup_{\varepsilon>0}\nu(T_{min,\varepsilon}, K_i)$. On rappelle que notamment que $\nu_i=0$ si et seulement si $K_i\nsubseteq H$. L'unicité de la forme logarithmique construite dans la démonstration précédente fournit immédiatement une version (semi) globale du lemme \ref{logloc}.

\begin{PROP}\label{formsemiglob} Sur un voisinage suffisamment petit $V$ de $K$, le feuilletage ${\mathcal F}_{V\setminus K}$ est définie par une section de ${\mathcal I}_{d\log}^\varepsilon$ qui s'étend en une forme logarithmique $\eta_K$ sur $V$   dont le résidu le long de $K_i$ est égal à $\nu_i +1$, $i=1,...,r$. Cette forme $\eta_K$ est par ailleurs unique.
\end{PROP}

Remarquons que le support $\vert (\eta_K) _\infty\vert$ du diviseur des pôles $(\eta_K) _\infty$ de $\eta_K$ peut être {\it a priori} plus gros que $K$ (par construction, les  pôles de $\eta_K$ sont exactement localisés sur les séparatrices locales le long de $K$). En termes d'intersection, ceci se reflète facilement de la façon suivante:
\begin{lemme}\label{lemautointersec} Soit $\theta$ une forme de Kähler sur $M$; pour tout $1\leq i\leq r$, posons 

$$a_i=\{K_i\}\sum_j(\mu_j+1)\{K_j\}{\{\theta\}}^{n-2}.$$
Alors, pour tout $i$, $a_i\leq 0.$

De plus, $a_i=0,\ i=1,...,r$ si et seulement si $\vert (\eta_K) _\infty\vert=K$.

\end{lemme}

Nous supposons par la suite que $K$ est de plus {\it connexe}.

\begin{PROP}\label{seplocale}La famille $\{K_1,....,K_r\}$ est exceptionnelle si et seulement si $\vert (\eta_K) _\infty\vert\varsupsetneq K$.
\end{PROP}

\begin{proof} Soit $(M,\{\theta\})$ une polarisation de  $M$ fournie par une forme de Kähler $\theta$. On peut d'abord rappeler (viii) proposition \ref{Z(alpha)nef}) que la famille $\{K_1,....,K_r\}$ est exceptionnelle si et seulement  si la matrice $(m_{ij}={K_i}{K_j}{\omega}^{n-2})$ est définie 
négative.

Supposons que $\vert (\eta_K) _\infty\vert=K$; on déduit alors du lemme que précédent que $\{D\}^2{\{\omega\}}^{n-2}=0$, en posant $D=\sum_i (\mu_i +1)D_i$; la famille n'est donc pas exceptionnelle.

Supposons inversement que $\{D_1,...,D_r\}$ n'est pas exceptionnelle; quitte à réordonner les indices, il existe $1\leq p\leq r$, $\mu_1,\mu_2,...,\mu_p>0$ tels que 

$${(\sum_{i=1}^p\mu_i\{D_i\})}^2=0$$

(conséquence de iii) proposition \ref{Z(alpha)nef}).

En appliquant à nouveau le lemme \ref{lemautointersec}, un simple calcul nous montre que tous les pôles de $\eta$ restreinte à un petit voisinage de $\tilde K:=K_1\cup K_2\cup...\cup K_p$ sont localisés sur $\tilde K$. Par hypothèse de connexité, on a en fait $\tilde K=K$ et donc bien $\vert (\eta_K) _\infty\vert= K$.

\end{proof}




\section{Propriétés de l'application développante}\label{unifo}

Les feuilletages étudiés sont transversalement hyperboliques ou euclidiens en dehors du support $H$ de la partie négative (en un sens large car les intégrales premières locales qui définissent la structure métrique transverse peuvent avoir un lieu critique non vide, i.e: $\mathcal F$ peut présenter des singularités en dehors de $H$).  Suivant la théorie générale des feuilletages transversalement homogènes, on peut donc considérer l'application développante $\rho$ de $\mathcal F$, définie sur le revêtement universel $N$ de $M\setminus H$. Cette application est en fait induite par le prolongement analytique d'une intégrale première admissible locale ( bien défini sur $N$) et est unique modulo l'action de $I^\varepsilon$. Rappelons en quelques propriétés (cf.\cite{god}):

$\rho$ est holomorphe sur $N$, à valeur dans $U^\varepsilon$  et submersive en dehors de $S=\pi^{-1}(Sing\ \mathcal F)$ où $\pi$ désigne le morphisme de revêtement. 

Les feuilles du feuilletage relevé $\tilde{ \mathcal F}={\pi}^*\mathcal F$ sont données par les composantes connexes des niveaux de $\rho_{\vert N\setminus S}$.

Il existe une représentation $r$ de $\pi_1(M\setminus H)$ dans $I^\varepsilon$ telle que pour tout $(\gamma,x)\in \pi_1(M\setminus H)\times N$, on ait
   $$\rho(\gamma .x)=r(\gamma)\rho (x).$$
   
   Rappelons enfin que la développante est   {\it complète} si l'image de $\rho(N)$ est $U^\varepsilon$.\\

   
   \begin{THM}\label{dévelop}  L'application développante $\rho$ associée aux feuilletage $\mathcal F$ est {\bf complète}. De plus, pour tout $c\in U^\varepsilon$, la fibre $\rho^{-1}(c)$ est connexe.

\end{THM}

Nous aurons besoin des deux résultats suivants.

\begin{lemme}Soit $\mathcal F$ un feuilletage holomorphe défini au voisinage de $0\in{\CC}^n$ amettant une intégrale première élémentaire  $$F=\prod_{i=1}^mf_i^{\gamma_i}.$$ Soit $X_1=\{f_1=0\},\ p\in X_1$ régulier et $T_p$ une petite tranversale au feuilletage en $p$; alors le saturé de $T_p$ par $\mathcal F$ contient un ouvert de la forme $W\setminus \{\prod_{i=1}^m=0\}$ où $W$ est un voisinage de l'origine. 
\end{lemme}

Ce lemme admet, via un simple argument de flow-box la version semi-globale suivante:

\begin{cor}\label{satsemiglob}
  Soit $\mathcal F$ un feuilletage holomorphe défini au voisinage d'une hypersurface invariante connexe $K$. On suppose $\mathcal F$ à singularités élémentaires sur $K$. .Soit $ p\in K$ un point régulier et $T_p$ une petite tranversale au feuilletage en $p$; alors le saturé de $T_p$ par $\mathcal F$ contient un ouvert de la forme $W\setminus K$ où $W$ est un voisinage de $K$. 

\end{cor}

  {\it Démonstration du théorème \ref{dévelop}}. Soit $g$ une métrique hermitienne sur $M$ et $V$ l'ouvert de $M\setminus H$ où $\mathcal F$ est régulier et où son fibré tangent $T\mathcal F$ est donc bien défini. Munissons ${T\mathcal F}^\perp$ de la métrique induite par celle qu'on a implanté sur $N_\mathcal F$ (i.e associée à la métrique euclidienne ou hyperbolique transverse). On récupère ainsi sur $V$ une métrique hermitienne qui préserve cette décomposition orthogonale (en général non holomorphe). 


 Il existe par conséquent 
une suite strictement croissante de  réels $(a_i),\ a_0=0,\ lim_{i\rightarrow +\infty}=a$  associée à des segments géodésiques (pour la métrique $g$) $\alpha_i:[a_i, a_{i+1}[\rightarrow V$ tel que pour chaque indice $i$ et tout $t\in [a_i, a_{i+1}[$, on ait 

    $$\rho \bigl ({\pi}^{-1}\alpha_i(t)\bigr )=\gamma (t).$$
Quitte à considérer une sous-suite si nécessaire, on peut de plus supposer que $p_i:=\alpha_i (a_i)$ converge vers $p\in H\cup Sing\mathcal F$. Notons ${\mathcal L}_i$ la feuille de $\mathcal F$ passant par $p_i$. On peut maintenant invoquer le lemme et corollaire précédents ainsi que la proposition \ref{seplocale} pour conclure qu'il existe $q\in V$ tel que pour tout voisinage $V_q$ de $q$, on ait ${\mathcal L}_i\cap V_q\not =\emptyset$ pour $i$ assez grand.
 Fixons $V_q$ et $i$ comme ci-dessus; soient $q_i\in V_q\cap {\mathcal L}_i$ et  $\beta:[0,1]\rightarrow {\mathcal L}_i$ un chemin joignant  $p_i$ à $q_i$ et $h_\beta$ l'application d'holonomie correspondante, dont la différentielle induit donc une isométrie entre  ${T_{p_i}\mathcal F}^\perp$ et  ${T_{q_i}\mathcal F}^\perp$.

Soit $v_i=dh_\beta ({\alpha_i}'(t_i)$; quitte à prendre $i$ suffisamment grand, on peut choisir $q_i$ de telle sorte que le segment géodésique $\alpha: [0,a_i]\rightarrow V_q, \alpha'(0)=v_i$ soit bien défini; ceci contredit visiblement le fait que $a_i\notin\rho(N)$.

\begin{rem}
 On vient en fait d'établir la propriété suivante:

\noindent il existe $r>0$ tel que sur toute feuille $\mathcal L$ de $\mathcal F$ en restriction à $M\setminus H$, il existe $q\in \mathcal L$ tel que la boule géodésique $B(q,r)$ (pour la métrique $g$) soit bien définie et tel que $\mathcal F$ admette une intégrale première admissible sur $B(q,r)$.
\end{rem}



En ce qui concerne la connexité, l'argument est assez similaire. Notons $\mathcal R$ la relation  d'appartenance à la même composante connexe des fibres de l'application développante $\rho$. Cette dernière induit sur $N/\mathcal R$ (muni de la topologie quotient) une application continue ${\rho}_\mathcal R$ surjective à valeur dans $U^\varepsilon$ qui est de plus un {\it homéomorphisme local}.

 Plus précisément, soit $T$ une transversale à
$\pi^*\mathcal F$ sur laquelle la développante $\rho$ est injective; soit $U_T$ l'ouvert de $N$ obtenu en saturant $T$ par $\mathcal R$. Il est alors clair que $\rho_\mathcal R$ définit un homéomorphisme entre les ouverts $U_T/\mathcal R$ et $\rho (T)$. 
cet homéomorphisme local  est en fait un revêtement par la remarque ci-dessus; les fibres de $\rho$ sont donc connexes par simple connexité de $U^\varepsilon$.

\qed

\section{Uniformisation et dynamique}\label{dyn}

La structure riemanienne transverse de nos feuilletages, telle qu'elle est précisée dans la section \ref{unifo}, peut éventuellement ne plus être définie sur un sous ensemble analytique qui contient le support $H$ de la partie négative de ${N_\mathcal F }^*$. Le fait que l'on ait un contrôle explicite de ce type de dégénerescence n'affecte pas profondément leur dynamique en comparaison du cas classique, i.e les feuilletages transversalement riemanniens sur une variété (réelle) compacte. Rappelons à cet effet quelques propriétés  de ces derniers (cf \cite{mo}).

-La variété est union disjointe des minimaux du feuilletage.

-Lorsque le feuilletage est de codimension 2 réelle, un minimal est de l'un des trois types suivants:soit une feuille compacte, soit la variété elle-même, soit une hypersurface réelle.


Nous réadoptons les notations de la section \ref{unifo}.

Soit $G$ l'image de la représentation $r$ et $G_0$ la composante neutre de son adhérence $\overline{G}$ dans le groupe $I^\varepsilon$.

Dans la mesure où nous manipulons des feuilletages à priori singulier (en particulier si $H\not=\emptyset$), il sera approprié de modifier légèrement la définition de feuille pour se ramener en quelque sorte au cas régulier.

A cet effet, soit $c\in U^\varepsilon$ et $F_c$ la fibre $p^{-1}(c)$. Posons $G_c=\pi (F_c)$; par construction $G_c=G_{c'}$ si et seulement si il existe $g\in G$ tel que $c'=g(c)$ et $G_c\cap G_{c'}=\emptyset$ sinon.

Il y a deux configurations possibles: 

-ou bien il existe $p\in Sing\ \mathcal F\cup H$ et une séparatrice locale $X_p$ de $\mathcal F$ en $p$ telle que $X_p\setminus H\subset G_c$, auquel cas on pose ${\mathcal L}_c=G_c\cup\mathcal C$, où $\mathcal C$ est la composante connexe de $H$ contenant $p$;

-ou bien il n'existe pas de tel $p$, auquel cas on pose ${\mathcal L}_c=G_c$. 

\begin{definition}
 Un sous ensemble de $M$ de la forme ${\mathcal L}_c$ est appelée {\bf feuille adaptée } du feuilletage $\mathcal F$.
\end{definition}

\begin{definition}
 Un sous-ensemble $\mathcal M$ de $M$ est appelé {\bf minimal adapté} du feuilletage si $\mathcal M$ est l'adhérence dans $M$ d'une feuille adaptée  et est minimal pour l'inclusion 
\end{definition}

Remarquons d'abord  que la collection $\{{\mathcal L}_c\}_{c\in U^\varepsilon}$ définit une partition de $M$; en effet, soit $X_p$, $X_{p'}$ deux germes de séparatrices de $\mathcal F$ localisées en deux point $p,p'$ d'une même composante connexe $\mathcal C$ de $H$ et telles que $X_p, X_{p'}$ ne coïncident pas avec une branche locale de $\mathcal C$. Au voisinage de $\mathcal C$; soit $\eta_\mathcal C$ la forme logarithmique produite par le théorème \ref{formsemiglob} sur un voisinage de $\mathcal C$.
Sur $V$, on peut choisir $m\in X_p\setminus \mathcal C$, $m'\in X_{p'}$ ainsi qu'un lacet $\gamma\in V\setminus\mathcal C$ joignant $m$ à $m'$. Puisque les résidus de $\eta$ le long de ces deux séparatrices locales valent $1$, on obtient au voisinage de $m$ et $m$ deux sections $F_m,F_{m'}$ de ${\mathcal I}^\varepsilon$ de la forme $e^{\int\eta}$ de telle sorte que $F_{m'}$ soit obtenu comme prolongement analytique de $F_m$ le long de $\gamma$. Puisque $F_m(m)=F_{m'}(m')=0$, on conclut que $X_p$ et $X_{p'}$ appartiennent à une même feuille adaptée.

On notera $M/\mathcal F$ l'espace quotient associé à cette partition. On obtient alors une bijection canonique $\Psi: M/\mathcal F\rightarrow U^\varepsilon/G$ induite par la correspondance ${\mathcal L}_c\rightarrow c$.

En combinant ce qui vient d'être dit avec le corollaire \ref{satsemiglob} on obtient l'énoncé suivant, qu'on peut interpréter comme une sorte d'uniformisation de l'espace des feuilles.

\begin{THM}
 L'application $\Psi$ ci-dessus est un homéomorphisme (pour la topologie quotient à la source et au but).
\end{THM}

Par compacité de $M$, on a en particulier le
\begin{cor}
 Le groupe $G$ est co-compact, i.e: $U^\varepsilon/G$ est compact (non nécessairement séparé) pour la topologie quotient 
\end{cor}

\begin{rem}
 Puisque toute les singularités du feuilletage sont de nature élémentaire, on pourra également noter que les feuilles adaptées qui ne  sont pas feuille au sens ordinaire sont en nombre fini 
\end{rem}

Au vu de ce qui précède, on peut observer que les propriétés dynamique du feuilletage sont fidèlement reflétés par celle de du groupe $G$. Il s'agit de décrire les diverses formes que peut présenter celui-ci (ou plus exactement son adhérence $\overline{G}$).

On procède d'abord par élimination en supposant   que $U^\varepsilon=\DD$ et $G$ affine (c'est-à-dire fixe un point sur le bord du disque). En particulier, $G$ ne contient pas de transformation elliptique non triviale. Il n'y a donc  pas d'obstruction à ce que  $F=e^{\int\eta_H}$ soit une intégrale première  {\it holomorphe} du feuilletage sur un voisinage $V$ de $H$. Par construction, $F_{|V\setminus H}\in{\mathcal I}^1 (V\setminus H)$ et le diviseur des zéros de $dF$ coïncide avec la partie négative $N$. En utilisant maintenant que $G$ est affine, on obtient finalement que $\mathcal F$ est défini par une section sans diviseur de zéro de ${N_\mathcal F}^*\otimes\mathcal O(-N)\otimes E$ où $E$ est un fibré en droite holomorphe {\it plat} (c'est peut-être plus évident à voir si l'on se place sur le modèle du demi-plan); ceci est incompatible avec le fait que la partie positive $Z$ est non triviale.

En expoitant par ailleurs la co-compacité de $G$ (et donc de $\overline{G}$), on aboutit facilement au descriptif complet qui suit.
\begin{prop}
 Le groupe $\overline {G}$ se présente sous l'une des formes exclusives suivantes (à conjugaison près dans ${\mathcal I}^\varepsilon$).\\

a) $G=\overline{G}$ et est donc un réseau co-compact.\\

b) L'action de $G$ est minimale et dans ce cas

$\varepsilon =0$ et $\overline {G}$ contient le sous groupe des translations de ${\mathcal I}^0$

ou 

$\varepsilon =1$ et $\overline G={\mathcal I}^1$.\\

c) $\varepsilon=0$ et $\overline{G}$ contient le groupe des translations réelles $T=\{t_\alpha, \ \alpha\in\RR\}$ ($t_\alpha (z)=z+\alpha)$ et plus précisément

$\overline G=\langle T,t\rangle\ (*)$

ou

$\overline G=\langle T,t,s\rangle\ (**)$

avec $t(z)=z+ai$ où $a$ est un réel fixé non nul et $s(z)=-z$.

\end{prop}

En synthétisant ce qui a été observé depuis la section \ref{unifo}, on aboutit à la description de la dynamique du feuilletage,  laquelle se décline suivant les différents cas identifiés dans la proposition précédente.

\begin{THM}

L'adhérence de toute feuille adaptée est un minimal adapté; de plus

dans le cas a), $\mathcal F$ est une fibration holomorphe à fibres connexes au dessus de la surface de Riemann compacte (orbifold) $U^\varepsilon/G$;

dans le cas b), $M$ est l'unique minimal adapté;

dans le cas c), chaque minimal adapté est une hypersurface analytique réelle.\footnote{Cet énoncé n'est pas tout à fait exact dans le cas (**) où l'on pourrait avoir quelques minimaux adaptés semi-analytiques, phénomène qu'on peut attribuer à  l'existence {\it a priori} de composantes de $H$ dont le coefficient correspondant dans la décomposition Zariski est un demi entier (ces composantes apparaissant alors comme le bord des minimaux en question). On peut toutefois se ramener au cas (*) au moyen d'un revêtement double (ramifié) associé à une section de ${({N_\mathcal F}^*)}^{\otimes 2}$) . }

\end{THM}

\begin{cor}
 Supposons que $\mathcal F$ n'est pas une fibration; alors l'ensemble des hypersurfaces irréductibles invariantes par le feuilletage forme une famille exceptionnelle En particulier, son cardinal ne peut excéder $\rho (M)$.
\end{cor}

\begin{proof}
 Raisonnons par l'absurde; il existe alors une hypersurface connexe $K=K_1\cup K_2...\cup Kr$ invariante par $\mathcal F$ et telle que la famille $\{K_1,..,K_r\}$ ne soit pas exceptionnelle.  En vertu de la proposition \ref{seplocale}, $K$ est nécessairement une feuille adaptée, ce qui contredit le théorème précédent. 
\end{proof}

\section{Le cas des feuilletages à classe canonique (numériquement) triviale \label{cantriviale}}

Soit $\mathcal F$  une distribution de codimension $1$ sur une variété Kähler compacte $M$ ($dim_\CC M=n$). On considère le fibré déterminant associé 

$$det\ \mathcal F={({\ \bigwedge}^{n-1}T_\mathcal F)}^{**}$$
qu'on peut donc voir comme sous faisceau inversible de ${\bigwedge}^{n-1}T_M$.

 On supposera en outre que \\

{\bf(*) on a la condition de trivialité numérique}

$$   c_1(\mathcal F):=c_1(det\ \mathcal F)=0$$
 \\

{\bf(**)le fibré canonique $K_M$ est pseudo-effectif.}\\

 Par adjonction, on obtient que 

$$K_M={(det\ \mathcal F)}^*\otimes {N_\mathcal F}^*$$

\noindent en conséquence de quoi le conormal ${N_\mathcal F}^*$ est numériquement équivalent  à $K_M$ et est donc {\it pseudo-effectif}. En réinvoquant le théorème \ref{de}, on peut conclure que $\mathcal F$ est intégrable. On désignera encore par $\mathcal F$ le feuilletage de codimension 1 associé. Il s'agit donc d'une sous classe de la famille de feuilletages considérés dans cet article; lorsque $\mathcal F$ est régulier on obtient dans \cite{to} une classification essentiellement complète ( pour faire bref, $\mathcal F$ provient, à revêtement fini près, d'un feuilletage linéaire sur le tore ou est une fibration isotriviale).\\

Comme l'indique le théorème ci-dessous, cette situation est en fait générale.

\begin{THM} Sous les hypothèses ci-dessus, le feuilletage $\mathcal F$ est nécessairement régulier.\footnote{Sous les hypothèses numériques (*) et (**), ce résultat persiste en fait pour des distributions de codimension quelconque (\cite{lpt}) }

\end{THM}

\begin{proof} Utilisons la forme positive $\eta_T=ie^{2\varphi} \omega\wedge \overline{\omega}$ introduite en section 1 (formule (\ref{eta})).
 
On rappelle que $\eta_T$ est fermée (au sens des courants) et permet donc de produire par dualité de Poincaré-Serre une classe non triviale $\{\beta\}\in H^{1,1}(M)$. Remarquons qu'on hérite par ailleurs d'une décomposition globale naturelle 

$$\eta_T=\alpha_1\wedge\alpha_2$$
où $\alpha_1=\omega$,
$\alpha_2=-ie^{2\varphi}  \overline{\omega}$ vus respectivement comme $(1,0)$ forme à valeur dans $N_\mathcal F$ et courant de bidegré $(0,1)$ à valeurs dans $N_\mathcal F^*$ (qui s'interprète comme forme linéaire continue sur les $(n,n-1)$ formes à valeurs dans  $N_\mathcal F$). 

On constate que $\alpha_1$ et $\alpha_2$ sont $\overline{\partial}$ fermées (c'est évident pour $\alpha_1$ et résulte par exemple de la formule (\ref{integ}) pour $\alpha_2$). Elles induisent donc, puisque $\eta_T$ est non nulle en cohomologie, des classes non triviales $\{\alpha_1\}\in  H_{\overline{\partial}}^{1,0}(N_\mathcal F),\ \{\alpha_2\}\in  H_{\overline{\partial}}^{0,1}(N_\mathcal F^*)$,

 Plus précisément, le cup produit 

$$\{\alpha_1\}\{\alpha_2\}\{\beta\}\in H_{\overline{\partial}}^{n,n}(M)$$ est non nul (il coïncide en effet avec $\{\eta_T\wedge \{\beta\}$).

Le théorème résulte alors d'un simple jeu de réécriture; en notant $c$ la classe de $\alpha_1\wedge\beta$ dans $H_{\overline{\partial}}^{0,n-1}(N_\mathcal F\otimes K_M*)$, on obtient que

$$\{\alpha_2\}c\not=0\label{alpha2cnot=0}$$
Par suite, on a nécessairement

$$\alpha_2\wedge\tilde c\not=0$$ 

où $\tilde c\in\Omega^{(0,n-1)}(E)$ est une forme lisse de bidegré $(o,n-1)$ à valeurs dans le fibré $E= N_\mathcal F\otimes K_M*$ représentant $c$.

D'après la formule d'ajonction précédente, on peut munir $E$ d'une métrique qui en fait {\it un fibré en droites hermitien plat}. 

Relativement à cette métrique (et une métrique kählerienne fixée sur $M$), on peut choisir $\tilde c$ {\it harmonique} Dans ce cas, par symétrie de Hodge, la forme conjuguée $\Omega:=\overline{\tilde c}$ est {\it holomorphe} et $\omega\wedge\Omega$ donne lieu à une section {\it non nulle} de $N_\mathcal F\otimes K_M\otimes E^*=\mathcal O_M$.

Cette section ne peut donc avoir de diviseur de zéros sur $M$, ce qui n'est visiblement possible que si $\mathcal F$ est en fait régulier.
\end{proof}

\begin{rem} On pourrait plus généralement remplacer (*) par l'hypothèse $det\ \mathcal F$ pseudo-effectif (ce qui est donc le cas pour $N_\mathcal F^*)$); lorsque $M$ est projective et $det\ \mathcal F$ n'est pas (numériquement) trivial, des résultats  de Miyaoka (voir par exemple \cite{sh}) garantissent que $M$ est uniréglée, ce qui est bien sûr incompatible avec l'hypothèse (**). Nous pensons que ce type d'obtruction persiste en kählerien mais nous ne connaissons pas de preuves.

\end{rem}

\end{document}